\documentclass[12pt,  reqno]{amsart}

\usepackage{amsthm,amsmath}

\newtheorem{theorem}{Theorem}[section]
\newtheorem{lemma}{Lemma}[section]
\newtheorem{proposition}{Proposition}[section]

\newtheorem{corollary}{Corollary}[section]

\newtheorem{remark}{Remark}

\numberwithin{equation}{section}
\theoremstyle{definition}
\theoremstyle{remark}

\begin{document}

\title{On an inequality of Andrews, De Lellis and Topping}
\author{Kwok-Kun Kwong}

\address{~School of Mathematical Sciences, Monash University, Victoria 3800, Australia.} \email{kwok-kun.kwong@monash.edu}

\thanks{Research partially supported by Australian Research Council Discovery Grant \#DP0987650 }

\renewcommand{\subjclassname}{\textup{2010} Mathematics Subject Classification}
\date{}


\begin{abstract}
Using the method of De Lellis-Topping \cite{lellis2012almost}, we prove some almost Schur type results. For example, one of our results gives a quantitative measure of how close the higher mean curvature of a submanifold is to its average value. We also derive another sharp Andrews-De Lellis-Topping type inequality involving the Riemannian curvature tensor and discuss its equality case.
\end{abstract}
\maketitle\markboth{}{}

\section{Introduction}
It is a classical result of Schur that if the Ricci curvature satisfies $Ric=\frac Rn g$ on a Riemannian manifold $(M^n,g)$, then the scalar curvature $R$ must be constant for $n\ge 3$. This is a simple consequence of the twice-contracted Bianchi identity $\mathrm{div}(Ric-\frac R2g)=0$. It is interesting to see how $R$ differs from a constant if $Ric-\frac Rng$ is close to zero. In this direction De Lellis and Topping proved in \cite{lellis2012almost} (and independently by Andrews, cf. \cite{chow2006hamilton} Corollary B.20) that
\begin{theorem}\label{thm: 1}
  If $(M^n,g)$ is a closed oriented Riemannian manifold $(n\ge 3)$ with nonnegative Ricci curvature, then
  $$\int_M (R-\overline R)^2 \le \frac{4n(n-1)}{(n-2)^2}\int_M |Ric - \frac Rng|^2.$$
  Here $\overline R$ is the average of $R$ on $M$. The equality holds if and only if $(M, g)$ is Einstein.
\end{theorem}

These types of stability and rigidity results have attracted a lot of attention in the last decade. In particular
the work of De Lellis-M{\"u}ller \cite{de2005optimal} started this new research field within the field of geometric analysis (we would like to thank the reviewer for pointing out this reference). In this paper, we will show an analogous result for the higher $r$-th mean curvature for closed submanifolds in space forms (Theorem \ref{thm: main}). In particular we will show that our result implies Theorem \ref{thm: 1} and some other results in \cite{CZ} and  \cite{C}. We will also show another version of this type of result which involves the Riemannian curvature tensor. More precisely, we prove the following Andrews-De Lellis-Topping type inequality:

\begin{theorem}\label{thm: 2}(Theorem \ref{thm: main}) 
Let $\Sigma^n$ ($n\ge 2$) be a closed immersed oriented submanifold in a space form $N^m$, $m>n$. Let $r\in \lbrace 1,\cdots, n-1\rbrace$. Assume that either
\begin{enumerate}
  \item $r$ is even, or
  \item  $r$ is odd and $N=\mathbb R^m$, or
  \item $\Sigma$ is of codimension one, i.e. a hypersurface.
\end{enumerate}
 Let $\overline H_r=\frac 1{ \mathrm{Area}(\Sigma)}\int_\Sigma H_r$ be the average of $H_r$  and $\stackrel{\circ}{T^{r}}= T^{r}- \frac{(n-r)}n H_r I$ be the traceless part of $T^{r}$. Let $\lambda$ be the first eigenvalue for the Laplacian on $\Sigma$ and suppose the Ricci curvature of $\Sigma $ is bounded from below by $-(n-1)K$, $K\ge 0$, then for $r=1,\cdots, n-1$, we have
\begin{equation*}
\int_\Sigma |H_r-\overline H_r|^2 \le \frac{ n(n-1)}{(n-r)^2} (1+\frac{nK}\lambda)\int_\Sigma |\stackrel \circ {T^{r}}|^2.
\end{equation*}
\end{theorem}

Here $H_r$ is the $r$-th mean curvature and $T^r$ is the $r$-th Newton transformation of the second fundamental form and will be defined in Section \ref{sec: submfd}. In particular, if $r=1$, then $-\stackrel \circ {T^1}= \stackrel\circ A$, the traceless second fundamental form. In particular, this recovers the classical result that if $\Sigma$ is a closed embedded totally umbilic hypersurface in $\mathbb R^{n+1}$, $\mathbb H^{n+1}$ or the hemisphere $\mathbb S^{n+1}_+$, then $H$ is constant and thus is a distance sphere by \cite{montiel1991compact}.

In Section \ref{sec: other}, we will also prove:
\begin{theorem}\label{thm: 3}(Theorem \ref{thm: R})
Suppose $(M^n,g)$ ($n\ge 3$) is a closed oriented Riemannian manifold such that its Ricci curvature is bounded from below by $-K$, $K\ge 0$, then we have
\begin{equation*}
\begin{split}
\int _M (R-\overline R) ^2&\stackrel{\mathrm{(i)}}\le \frac {4n(n-1) } {(n-2)^2}(1+\frac{nK}\lambda) \int_M |Ric-\frac Rn g|^2 \\
&\stackrel{\mathrm{(ii)}}\le \frac {n(n-1) } {n-2} (1+\frac{nK}\lambda)\int_M |Rm-\frac {R }{n(n-1)} B|^2,
\end{split}
\end{equation*}
where $\overline R$ is the average of its scalar curvature $R$, $\lambda$ is the first eigenvalue for the Laplacian on $M$  and $B_{ijkl}=g_{ik}g_{jl}-g_{il}g_{jk}$ is the curvature tensor with curvature $1$. The equality sign in (i) holds if and only if $(M,g)$ is Einstein. If $n\ge 4$, then the equality sign in (ii) holds if and only if $(M,g)$ is locally conformally flat. Both (i) and (ii) become equalities if and only if $(M,g)$ has constant curvature.

\end{theorem}

This result gives a quantitative version of another result of Schur: if $(M^n,g)$ $(n\ge 3)$ has sectional curvature which depends on its base point only, then its curvature is constant.

The rest of this paper is organized as follows. In Section \ref{sec: submfd}, we derive Theorem \ref{thm: 2} for closed oriented submanifolds (not necessarily of codimension one) in space forms. In Section \ref{sec: lovelock}, we discuss the relation between Theorem \ref{thm: 2} and a corresponding result in \cite{geproblems}, and show that indeed our result implies Theorem \ref{thm: 1} and a result in \cite{geproblems}. In Section \ref{sec: equality} we discuss the equality case of Theorem \ref{thm: main}. Finally in Section \ref{sec: other} we will prove Theorem \ref{thm: 3} and show that the constants in the inequalities are optimal.

{\sc Acknowledgments}:
The author would like to thank Gilbert Weinstein for stimulating discussions and useful comments.

\section{Higher mean curvatures of submanifolds in space forms}\label{sec: submfd}

Let $\Sigma^n$ be an immersed submanifold in a Riemannian manifold $(N^m,h)$, $n<m$.
The second fundamental form of $\Sigma $ in $N $ is defined by $A(X,Y)=-({\overline \nabla _XY})^\perp$ and is normal-valued. Here $\overline \nabla $ is the connection on $N$. We denote $A(e_i, e_j)$ by $A_{ij}$, where $\lbrace e_i\rbrace_{i=1}^n$ is a local orthonormal frame on $\Sigma$.

We define the $r$-th mean curvature as follows. If $r$ is even,
$$ H _r =  \frac 1{r!}\sum_{\substack{i_1,\cdots, i_r\\
j_1, \cdots, j_r}} \epsilon_{j_1\cdots j_r}^{i_1\cdots i_r}h(A_{i_1j_1},A_{i_2j_2})\cdots h(A_{i_{r-1}j_{r-1}},A_{i_rj_r}).$$

If $r$ is odd, the $r$-th mean curvature is a normal vector field defined by
$$ H _r =\frac 1{r!}\sum_{\substack{i_1,\cdots, i_r\\
j_1, \cdots, j_r}} \epsilon _{j_1\cdots j_r}^{i_1\cdots i_r}h(A_{i_1j_1},A_{i_2j_2})\cdots h(A_{i_{r-2}j_{r-2}},A_{i_{r-1}j_{r-1}})A_{i_rj_r}.$$
Here $\epsilon_{i_1 \cdots i_r}^{j_1\cdots j_r}$ is  zero if $i_k=i_l$ or $j_k=j_l$ for some $k\ne l$ , or if $\lbrace i_1, \cdots, i_r\rbrace \ne \lbrace j_1, \cdots, j_r\rbrace$ as sets, otherwise it is defined as the sign of the permutation $(i_1, \cdots, i_r)\mapsto (j_1, \cdots, j_r)$.

In the codimension one case, i.e. $\Sigma$ is a hypersurface, by taking the inner product with a unit normal if necessary, we can assume $H_r$ is scalar-valued. In this case the value of $H_r$ is given by
\begin{equation}\label{eq: codim 1}
H_r =\sum_{i_1<\cdots< i_r}k_{i_1}\cdots k_{i_r}
\end{equation}
where $\{k_i\}_{i=1}^n$ are the principal curvatures. This definition of $H_r$ will be used whenever $\Sigma $ is a hypersurface.

Following \cite{grosjean2002upper} and \cite{reilly1973variational}, we define the (generalized) $r$-th Newton transformation $T^r$ of $A$ (as a $(1,1)$ tensor, possibly vector-valued) as follows.\\
If $r$ is even,
$$ {(T^r)}_j^{\,i}= \frac 1 {r!}
\sum_{\substack{i_1,\cdots, i_r\\ j_1, \cdots, j_r}}
\epsilon^{i  i_1 \ldots  i_r}_{j  j_1 \ldots  j_r}
h(A_{i_1j_1},A_{i_2j_2})\cdots h(A_{i_{r-1}j_{r-1}},A_{i_rj_r}).
$$
If $r$ is odd,
$${(T^r)}_j^{\,i}= \frac 1 {r!}
\sum_{\substack{i_1,\cdots, i_r\\ j_1, \cdots, j_r}}
\epsilon^{i  i_1 \ldots  i_r}_{j  j_1 \ldots  j_r}
h(A_{i_1j_1},A_{i_2j_2})\cdots h(A_{i_{r-2}j_{r-2}},A_{i_{r-1}j_{r-1}})A_{i_rj_r}.
$$
Again, in the codimension one case, we can assume $T^r$ is an ordinary $(1,1)$ tensor and if $\lbrace e_i\rbrace_{i=1}^n$ are the eigenvectors of $A$, then
$$T^r(e_i)=\frac 1{{n\choose r}}\sum_{\substack{i_1<\cdots <i_r\\
i\ne i_l}}k_{i_1}\cdots k_{i_r}e_i.$$

This definition of $T^r$ will be used whenever $\Sigma$ is a hypersurface.

\begin{lemma}\label{lem: div}
If $\Sigma^n$ is an immersed  submanifold in a space form $N^m$, let $\mathrm{div}$ and $\mathrm{tr}$ denotes the divergence and the trace on $\Sigma$ respectively, then
  $$\mathrm{div} (T^r)=0 \textrm{\quad  and \quad}\mathrm{tr}(T^{r})= (n-r)H_r.$$
\end{lemma}
\begin{proof}
  The first assertion follows from the proof of \cite{grosjean2002upper} Lemma 2.1. But since it has only been shown in the case where $r$ is even in that paper, let us assume the odd case, and for the sake of demonstration let $r=3$. Using local orthonormal frame, the first assertion follows from
\begin{equation*}
  \begin{split}
   &r! \nabla _i{ (T^{r})}_{i}^j\\
   =&  \epsilon_{i \, i_1 i_1 i_3}^{j  \,j_1 j_2 j_3}
   (h(\nabla_iA_{i_1j_1 }, A_{i_2j_2 }) A_{i_3j_3 }+h(A_{i_1j_1 }, \nabla_iA_{i_2j_2 }) A_{i_3j_3 }+h(A_{i_1j_1 }, A_{i_2j_2 }) \nabla_i A_{i_3j_3 })\\
    =&  \epsilon_{i \, i_1 i_1 i_3}^{j  \,j_1 j_2 j_3}
   (h(\nabla_{i_1}A_{ij_1 }, A_{i_2j_2 }) A_{i_3j_3 }+h(A_{i_1j_1 }, \nabla_{i_1}A_{ij_2 }) A_{i_3j_3 }+h(A_{i_1j_1 }, A_{i_2j_2 }) \nabla_{i_3} A_{ij_3 })\\
   =&  -\epsilon_{i \, i_1 i_1 i_3}^{j  \,j_1 j_2 j_3}
   (h(\nabla_iA_{i_1j_1 }, A_{i_2j_2 }) A_{i_3j_3 }+h(A_{i_1j_1 }, \nabla_iA_{i_2j_2 }) A_{i_3j_3 }+h(A_{i_1j_1 }, A_{i_2j_2 }) \nabla_i A_{i_3j_3 })\\
    =&-r! \nabla _i (T^{r})_{i}^j
  \end{split}
\end{equation*}
where we have used the Codazzi equation $\nabla _i A_{jk}=\nabla _j A_{ik}$ (as $N$ has constant curvature).
The second assertion is straightforward.
\end{proof}

By Lemma \ref{lem: div} we immediately have the following Schur-type theorem which is perhaps well-known to experts (we write $V s$ instead of $V\otimes s$ for a vector-valued function $V$ and a $(1,1) $ tensor $s$):
\begin{theorem}
Let $\Sigma^n$ be a closed (compact without boundary) immersed submanifold in a space form $N^m$ and $r\in\lbrace1, \cdots, n-1\rbrace$. If the traceless part $\stackrel{\circ}{T^{r}}$ of $T^{r}$ vanishes, i.e. $T^r= \frac {n-r}nH_r I$, then $H_r$ is parallel. (In particular it is constant whenever it can be defined as a scalar. )
\end{theorem}
\begin{proof}

   By Lemma \ref{lem: div}, as $\stackrel \circ {T^r} = T^r- \frac{(n-r)}n H_r I$, we have
$0=\mathrm{div}(\stackrel\circ {T^r})=-\frac {n-r}n\nabla H_r $. i.e. $H_r$ is parallel.
\end{proof}
\begin{theorem}\label{thm: main}
Let $\Sigma^n$ ($n\ge 2$) be a closed immersed oriented submanifold in a space form $N^m$, $m>n$. Let $r\in \lbrace 1,\cdots, n-1\rbrace$. Assume that either
\begin{enumerate}
  \item $r$ is even, or
  \item  $r$ is odd and $N=\mathbb R^m$, or
  \item $\Sigma$ is of codimension one, i.e. a hypersurface.
\end{enumerate}
 Let $\overline H_r=\frac 1{ \mathrm{Area}(\Sigma)}\int_\Sigma H_r$ be the average of $H_r$ (which is a vector when $r$ is odd and is defined by \eqref{eq: codim 1} in the codimension one case) and $\stackrel{\circ}{T^{r}}= T^{r}- \frac{(n-r)}n H_r I$ be the traceless part of $T^{r}$. Let $\lambda$ be the first eigenvalue for the Laplacian on $\Sigma$ and suppose the Ricci curvature of $\Sigma $ is bounded from below by $-(n-1)K$, $K\ge 0$, then for $r=1,\cdots, n-1$, we have
\begin{equation}\label{eq: main}
\int_\Sigma |H_r-\overline H_r|^2 \le \frac{ n(n-1)}{(n-r)^2} (1+\frac{nK}\lambda)\int_\Sigma |\stackrel \circ {T^{r}}|^2,
\end{equation}
or equivalently,
$$ \int_\Sigma |T^{r} - \frac{n-r}n \overline H_r I|^2 \le n (1+\frac{(n-1)K}\lambda) \int_\Sigma |T^{r} - \frac{n-r}n H_r I|^2.$$
\end{theorem}

\begin{proof}
  We follow the ideas in \cite{lellis2012almost} and \cite{CZ}.  We do the case where $r$ is odd (and thus $N=\mathbb R^m$) first. We can assume $H_r-\overline H_r$ is not vanishing everywhere, otherwise there is nothing to prove.
  Let $F=(F^1,\cdots, F^m)$ be the solution to
  \begin{equation}\label{eq: delta}
    \begin{cases}
      \Delta F = H_r- \overline H_r\\
      \int_\Sigma F =0.
    \end{cases}
  \end{equation}
The solution exists because $\int_\Sigma H_r - \overline H_r=0$.
  As $\stackrel \circ {T^r} = T^r- \frac{(n-r)}n H_r I$ and $\textrm{div}(T^r)=0$ by Lemma \ref{lem: div}, (as a vector-valued $1$-form) we have
$$ \mathrm{div}(\stackrel\circ {T^r})=-\frac {n-r}n\nabla H_r.$$
Let us denote the dot product in $\mathbb R^m$ by $\cdot$ and the intrinsic inner product on $\Sigma $ by $\langle \cdot, \cdot\rangle$. Then
\begin{equation}\label{eq: sigma}
  \begin{split}
    \int_\Sigma |\Delta F|^2
    = \int_\Sigma (H_r- \overline H_r)\cdot\Delta F
    &= -\int_\Sigma \langle \nabla H_r,\nabla F\rangle\\
    &=\frac n{n-r} \int_\Sigma \langle \mathrm{div}(\stackrel \circ {T^r}),\nabla F\rangle\\
    &= \frac n{n-r}\int_\Sigma \langle -\stackrel \circ {T^r} , \nabla ^2F\rangle\\
    &= \frac n{n-r}\int_\Sigma \langle -\stackrel \circ {T^r} , \nabla ^2F- \frac {\Delta F }nI\rangle\\
    &\le\frac n{n-r} \|\stackrel \circ {T^r} \|_{L^2}\|\nabla ^2 F- \frac{\Delta F}n I\|_{L^2}.
  \end{split}
\end{equation}
We have
  \begin{equation}\label{eq: norm}
\begin{split}
  \|\nabla ^2 F -\frac {\Delta F}n I\|_{L^2}^2
  &= \int_\Sigma |\nabla ^2 F|^2 +\frac 1n \int_\Sigma |\Delta F|^2 -\frac 2n \int_\Sigma |\Delta F|^2\\
  &= \int_\Sigma |\nabla ^2 F|^2 -\frac 1n \int_\Sigma  |\Delta F|^2.
\end{split}
  \end{equation}
By Bochner formula, we have
\begin{equation}\label{eq: Bochner}
  \begin{split}
    \int_\Sigma |\nabla ^2F|^2
    &= \int_ M |\Delta F|^2 -\int_\Sigma Ric(\nabla F, \nabla F)\\
    &\le \int_\Sigma |\Delta F|^2 +(n-1)K\int_\Sigma |\nabla F|^2.
  \end{split}
\end{equation}
Here $Ric(\nabla F,\nabla F)=\displaystyle \sum_{\substack{1\le i,j\le n\\ 1\le l\le m}} R_{ij}\nabla _iF^l \nabla _j F^l$.
Consider
\begin{equation*}
  \begin{split}
    \int_\Sigma |  \nabla F|^2
    =-\int_\Sigma F \cdot\Delta F
    &\le (\int_\Sigma |F|^2)^\frac 12\left(\int_\Sigma  |\Delta F|^2\right)^\frac12\\
    &\le \left(\frac {\int_\Sigma |  \nabla F|^2}{\lambda}\right)^\frac 12\left(\int_\Sigma   |\Delta F|^2\right)^\frac12.
  \end{split}
\end{equation*}
Thus
\begin{equation}\label{ineq: lambda}
\int_\Sigma |\nabla F|^2 \le \frac 1{\lambda} \int_\Sigma  |\Delta F|^2.
\end{equation}
Here we have used the fact that the first eigenvalue $\lambda=\min \{ \frac {\int_\Sigma | \nabla \phi|^2}{\int_\Sigma \phi^2}: \int_\Sigma \phi=0, \phi\ne0\}$.
In view of \eqref{eq: Bochner} and \eqref{ineq: lambda}, \eqref{eq: norm} becomes
\begin{equation}\label{eq: Ric}
  \begin{split}
  \|\nabla ^2 F -\frac {\Delta F}n I\|_{L^2}^2
&\le (\frac{n-1}n) (1+\frac{nK}\lambda)\int_\Sigma |\Delta F|^2 .
  \end{split}
\end{equation}
Substitute this into \eqref{eq: sigma}, we obtain \eqref{eq: main}:
$$
\int_\Sigma |H_r-\overline H_r|^2 \le \frac{ n(n-1)}{(n-r)^2} (1+\frac{nK}\lambda)\int_\Sigma |\stackrel \circ {T^{r}}|^2.
$$
As $T^{r}- \frac{n-r}n \overline H_r I= \stackrel\circ {T^{r}}+\frac{n-r}n (H_r-\overline H_r)I$, by Pythagoras theorem we have
$$|T^{r}- \frac{n-r}n \overline H_r I |^2= |\stackrel \circ {T^{r}}|^2 +\frac {(n-r)^2}n|H_r-\overline H_r|^2.$$
Therefore \eqref{eq: main} can be rephrased as
$$ \int_\Sigma |T^{r} - \frac{n-r}n \overline H_r I|^2 \le n (1+\frac{(n-1)K}\lambda) \int_\Sigma |T^{r} - \frac{n-r}n H_r I|^2.$$

For the remaining cases where $r$ is even or $\Sigma$ is a hypersurface, as $H_r$ is scalar valued, just replace $F$ in \eqref{eq: delta} by a scalar valued function $f$ and apply the same
argument, we can get the result.
\end{proof}

\begin{remark}
  When $r=1$ and $\Sigma$ is a hypersurface, as $T^1= H_1 I-A$, it is easy to see that $\stackrel \circ {T^{1}}= -\stackrel \circ A$ where $\stackrel \circ A$ is the traceless part of the second fundamental form $A$. This generalizes \cite{Perez} Theorem 3.1 (see also \cite{CZ}). In the codimension one case, this recovers \cite{C} Theorem 1.10.
\end{remark}

\begin{remark}
  For an embedded hypersurface in Euclidean space, having nonnegative Ricci curvature is equivalent to $A\ge 0$ (i.e. convex), see \cite[p. 48]{Perez}. So when $K=0$, the curvature assumptions in Theorem \ref{thm: main} can be replaced by $\Sigma$ being convex when it is an embedded hypersurface.
\end{remark}

\section{Relations with Lovelock curvatures}\label{sec: lovelock}

In this section, we investigate the relation between Theorem \ref{thm: main} and an analogous result of Ge-Wang-Xia \cite{geproblems}.
Following \cite{geproblems}, we define the Lovelock curvatures (or the so called $2k$-dimensional Euler density in Physics) of a Riemannian manifold $(M^n,g)$, $k<\frac n2$, by
\begin{equation}
  R^{(k)}= \frac 1{2^k} \epsilon_{ i_1 \ldots  i_{2k}}^{j_1 \ldots  j_{2k}}{R_{j_1 j_2}}^{i_1i_2}\cdots {R_{j_{2k-1} j_{2k}}}^{i_{2k-1} i_{2k}}.
\end{equation}
We use the convention that $R_{ijij}$ is the sectional curvature. It can be easily seen that $R^{(1)}$ is the scalar curvature. We also define the generalized Einstein $2$-tensor by defining
$$ {E^{(k)}}_{i}^j=\frac 1{2^{k+1}} \epsilon_{ ii_1 \ldots  i_{2k}}^{jj_1 \ldots  j_{2k}}{R_{j_1 j_2}}^{i_1i_2}\cdots {R_{j_{2k-1} j_{2k}}}^{i_{2k-1} i_{2k}}.$$
We have the following analogue of Lemma \ref{lem: div} for $E^{(k)}$:
\begin{lemma}\label{lem: div3}
We have
$$ \mathrm{tr}(E^{(k)})= \frac{n-2k}2R^{(k)}\textrm{\quad and \quad }\nabla ^i {E^{(k)}}_i^j=0.$$
\end{lemma}
\begin{proof}
  The first assertion is a straightforward calculation. For the second assertion, for the sake of demonstration let $k=1$. Then this follows from
  \begin{equation*}
    \begin{split}
    2^{k+1}\nabla ^i {E^{(k)}}_i^j
    &=\epsilon_{i  i_1 i_2}^{j j_1j_2}\nabla^i{R_{j_1 j_2}}^{i_1i_2}\\
    &=-\epsilon_{i  i_1 i_2}^{j j_1j_2}(\nabla^{i_1}{R_{j_1 j_2}}^{i_2i}+\nabla^{i_2}{R_{j_1 j_2}}^{ii_1})\\
    &=-\epsilon_{i_2  i i_1 }^{j j_1j_2}\nabla^i{R_{j_1 j_2}}^{i_1i_2}-\epsilon_{i_1i_2i }^{j j_1j_2}\nabla^i{R_{j_1 j_2}}^{i_1i_2}\\
    &=-2\epsilon_{i  i_1 i_2}^{j j_1j_2}\nabla^i{R_{j_1 j_2}}^{i_1i_2}.
    \end{split}
  \end{equation*}
  Here we have used the Bianchi identity in the second line.
\end{proof}
We see that $E^{(k)}$ is divergence free and indeed $E^{(1)}$ is the Einstein tensor. By Lemma \ref{lem: div3}, it is clear that we can prove the analogue of Theorem \ref{thm: main} in this setting. Indeed, by using the same method, Ge, Wang and Xia \cite{geproblems} proved that (they have assumed $Ric\ge 0$, but their result can be easily extended to the version below):
\begin{theorem}[\cite{geproblems} Theorem 4]\label{thm: ge}
  Let $(M,g)$ be a closed oriented Riemannian manifold with $Ric\ge -(n-1)K$, $K\ge0$ and $1\le k<\frac n2$, then
  $$ \int_M |R^{(k)}-\overline R^{(k)}|^2 \le \frac{4n(n-1)}{(n-2k)^2}(1+\frac{nK}\lambda)\int_M |\stackrel{\circ}{E^{(k)}}|^2.$$
  Here $\overline R^{(k)}$ is the average of $R^{(k)}$.
\end{theorem}
We will show that Theorem \ref{thm: main} is equivalent to Theorem \ref{thm: ge} in the case where $r$ is even.
\begin{proposition}\label{prop: E=T}
For an immersed submanifold  $\Sigma\subset \mathbb R^{n+1}$, for $k=1, \cdots, \lfloor \frac n2\rfloor$,
  we have $$ E^{(k)}=\frac { (2k)!}2T^{2k}.$$
\end{proposition}

\begin{proof}
We denote the dot product in $\mathbb R^m$ by $\cdot$. Using a local orthonormal frame, by Gauss equation, we have
$${R_{ij}}^{kl}= A_{ik} \cdot A_{jl}-A_{il} \cdot A_{jk}.$$
Thus
\begin{equation*}
  \begin{split}
    &\epsilon_{i i_1 \ldots  i_{2k}}^{jj_1 \ldots  j_{2k}}{R_{j_1 j_2}}^{i_1i_2}\cdots {R_{j_{2k-1} j_{2k}}}^{i_{2k-1} i_{2k}}\\
    =&\epsilon_{i i_1 \ldots  i_{2k}}^{jj_1 \ldots  j_{2k}}(A_{i_1j_1}  \cdot A_{i_2j_2}-A_{i_2j_1}  \cdot A_{i_1j_2}  )\cdots (A_{i_{2k-1}j_{2k-1}}\cdot A_{i_{2k}j_{2k}}-A_{i_{2k}j_{2k-1}}\cdot A_{i_{2k-1}j_{2k}})\\
    =&2^k\epsilon_{i i_1 \ldots  i_{2k}}^{jj_1 \ldots  j_{2k}}(A_{i_1j_1}\cdot A_{i_2j_2})\cdots (A_{i_{2k-1}j_{2k-1}}\cdot A_{i_{2k}j_{2k}})\\
    =&2^k(2k)!{(T^{2k})}_i^{\,j}.
  \end{split}
\end{equation*}
This implies the result.
\end{proof}
\begin{proposition}\label{prop: equiv}
Theorem \ref{thm: main} is equivalent to
Theorem \ref{thm: ge} when $r=2k$ and $N=\mathbb R^m$.
\end{proposition}
\begin{proof}
To see that Theorem \ref{thm: main} is equivalent to Theorem \ref{thm: ge} when $r$ is even, we just observe that by Proposition \ref{prop: E=T}, clearly Theorem \ref{thm: ge} implies our result in the even case and when $N$ is Euclidean. On the other hand, since any manifold can be isometrically embedded into some $\mathbb R^m$ for $m$ large enough \cite{nash1956imbedding}, we see that our result also implies Theorem \ref{thm: ge}.
\end{proof}

\section{Equality case of Theorem \ref{thm: main}}\label{sec: equality}


It is easy to see from the proof that if the Ricci curvature assumption in Theorem \ref{thm: main} is strengthened to $Ric>-(n-1)Kg$, then $H_r=\overline H_r$ (as $F$ is constant). On the other hand, it is more subtle if we omit this assumption and so far we have only got some partial results.

The equality case for $r=1$ and when $\Sigma$ is an immersed hypersurface in the Euclidean space $\mathbb R^m$, the hyperbolic space $\mathbb H^m$ or the hemisphere $\mathbb S^m_+$, has been considered in \cite{CZ}, in which they prove that $\Sigma$ is a distance sphere. It seems that their proof  cannot be modified directly for our case because in their proof it is essential that $\Sigma$ contains a point whose Ricci curvature is positive, which is not true for submanifold in higher codimension in general. However it is easy to modify their proof with an additional assumption:

\textbf{The case for $r=1$ and  $Ric>-(n-1)Kg$ at one point. }\\
Suppose the equality case holds, then from \eqref{eq: sigma}, \eqref{eq: Bochner} and \eqref{eq: Ric} we know that $T^1-\frac{n-1}nH_1 I$ and $\nabla ^2 F -\frac{\Delta F}n I$ are linearly dependent and $Ric(\nabla F, \nabla F)+(n-1)K|\nabla F|^2=0$.

Suppose $\nabla ^2 F- \frac{\Delta F }n I=0$, then from \eqref{eq: sigma} we have $H_1=\overline H_1$. Otherwise there exists a constant $\mu$ such that
\begin{equation}\label{eq: mu}
 T^1-\frac{n-1}n H_r I= \mu (\nabla ^2 F-\frac {\Delta F}nI).
\end{equation}
%
Since $Ric+(n-1)Kg>0$ at $p$, $F$ is constant in a neighborhood of $p$.
By \eqref{eq: mu}, $T^r- \frac{n-1}nH_1 I=0$ and thus $H_1=\overline H_1$ is constant in a neighborhood of $p$. Suppose $F$ is not constant on $\Sigma$, then there is a smooth curve $\gamma:\mathbb R\to \Sigma$
with $\gamma(0)=p$ such that $0<t_0=\inf \lbrace t>0: F\circ\gamma\textrm{ is constant on }[0,t]\rbrace<\infty$. Let $q=\gamma(t_0)$.
In view of \eqref{eq: mu} and by continuity, at $q$, we have
\begin{equation}\label{eq: Ric(q)}
 T^1= \frac{n-1}n H_1 I = \frac{n-1}n \overline H_1 I .
\end{equation}

On the other hand, it is not hard to see that $T^1=H_1 I-A$ (i.e. $(T^1)_i^j= H_1\delta_i^j-A_{ij}$) and so by Gauss equation ($c$ is the curvature of $N$), 
\begin{equation*}
  \begin{split}
R_i^j
&= (n-1)c\delta_i^j +h(H_1, A_{ij})-\sum_k h(A_{ik}, A_{kj})\\
&= (n-1)c\delta_i^j +\sum_k h((T^1)_i^k, A_{kj})\\
&= (n-1)c\delta_i^j +\sum_k h((T^1)_i^k, H_1\delta_k^j -{(T^1)}_{k}^j).
  \end{split}
\end{equation*}

 So by \eqref{eq: Ric(q)}, we have $R_i^j(q)+(n-1)K\delta_i^j= R_i^j(p)+(n-1)K\delta _i^j$ and in particular it is positive definite at $q$. This implies $F\circ \gamma$ is constant near $t=t_0$, a contradiction. We conclude that $H_1$ is constant. In the particular case where $\Sigma$ is an embedded hypersurface in $N=\mathbb R^m$, $\mathbb H^m$ or $\mathbb S^m_+$, as $T^1= H_1 I-A$, it is easy to see that $A= \frac {H_1}ng=\frac{\overline H_1}ng$, i.e. $\Sigma$ is totally umbilic and so is a geodesic hypersphere in $N$.

\textbf{The case for $r=2$. }\\
Suppose $\Sigma^n$ is immersed in a space form $N^m$ of curvature $c$, using the same computations as in the proof of Proposition \ref{prop: E=T}, we can get
\begin{equation}\label{eq: E}
  E^{(1)}= T^2+ {{n-2}\choose 2 }cI.
\end{equation}
Here ${l\choose k}=0$ if $k>l$. On the other hand, it is not hard to see that $E^{(1)}= Ric-\frac R2 g$ is the Einstein tensor (one way of seeing that without computing is to observe that $E^{(1)}$ is a $2$-tensor which contains only linear term involving the curvature and is divergence free, thus, up to constant, it must be the Einstein tensor). Thus
$$\stackrel{\circ }{T^2}=\stackrel\circ{Ric}=Ric -\frac Rn I.$$
 Note also that $R^{(1)}= R$ and is equal to $H_2$ up to an additive constant which only depends on $c$ and $n$. Thus our result in this case is reduced to Theorem 1.2 of \cite{cheng2012generalization} or the $k=1$ case of Theorem \ref{thm: ge}. By the rigidity case of \cite{cheng2012generalization} Theorem 1.2, we deduce that the equality holds in Theorem \ref{thm: main} in the $r=2$ case if and only if $\Sigma$ is Einstein. Therefore $R$ is constant and so is $H_2$ by Gauss equation.  In particular, if $\Sigma$ is an embedded hypersurface in $\mathbb R^m$, $\mathbb H^m$ or the hemisphere $\mathbb S^m_+$, then by \cite[Theorems 4,7 and 10]{montiel1991compact} it is a geodesic hypersphere.

 Let us summarize the known results:
 \begin{theorem}
With the assumption as in Theorem \ref{thm: main}, suppose the inequality in \eqref{eq: main} becomes an equality. Assume either
\begin{enumerate}
  \item
  $r=1$ and $\Sigma$ is an immersed hypersurface in $\mathbb R^{n+1}$, $\mathbb H^{n+1}$ or $\mathbb S^{n+1}_+$, or
  \item
  $r=1$ and $Ric>-(n-1)Kg$ at one point, or
  \item
  $r=2$, or
  \item
  $1\le r\le n-1$ and  $Ric>-(n-1)Kg$.
\end{enumerate}
Then
$H_r=\overline H_r$ is constant. In the case where $\Sigma$ is an embedded hypersurface in $\mathbb R^m$, $\mathbb H^m$ or $\mathbb S^m_+$ (the hemisphere), then $\Sigma$ is a geodesic hypersphere.
 \end{theorem}

\section{Another form of almost-Schur type theorem}\label{sec: other}

In this section, we derive another form of Schur-type theorem, which gives a quantitative version of the following classical result of Schur: if $(M^n,g)$ $(n\ge 3)$ has sectional curvature which depends on its base point only, then its curvature is constant.

\subsection{Main result}

We first set up some notations. Let $\mathcal T^r(M)$ denote the space of covariant $r$-tensor on $M$ (e.g. $g\in \mathcal T^2(M)$).
The Kulkarni-Nomizu product $\odot: \mathcal T^2(M)\times \mathcal T^2(M)\to \mathcal T^4(M)$
is defined by (see e.g. \cite[p.47]{besse2008einstein})
\begin{equation*}
  \begin{split}
(\alpha\odot \beta)(X,Y,Z,W)
=& \alpha(X,Z)\beta(Y,W)+\alpha (Y,W)\beta(X,Z)\\
&-\alpha(X,W)\beta(Y,Z)-\alpha(Y,Z)\beta(X,W).
  \end{split}
\end{equation*}
We define $B= \frac 12 g\odot g$. It is easy to see that $B$ is the Riemann curvature tensor of a space form with curvature $1$ (we use the convention that $R_{ijij}$ is the sectional curvature). In local coordinates, it is given by
\begin{equation*}\label{eq: B}
B_{ijkl}=g_{ik}g_{jl}-g_{il}g_{jk}.
\end{equation*}

\begin{theorem}\label{thm: R}
Suppose $(M^n,g)$ ($n\ge 3$) is a closed oriented Riemannian manifold such that its Ricci curvature is bounded from below by $-(n-1)K$, $K\ge 0$, then we have
\begin{equation}\label{ineq: R}
\begin{split}
\int _M (R-\overline R) ^2&\stackrel{\mathrm{(i)}}\le \frac {4n(n-1) } {(n-2)^2}(1+\frac{nK}\lambda) \int_M |Ric-\frac Rn g|^2 \\
&\stackrel{\mathrm{(ii)}}\le \frac {n(n-1) } {n-2} (1+\frac{nK}\lambda)\int_M |Rm-\frac {R }{n(n-1)} B|^2,
\end{split}
\end{equation}
where $\overline R$ is the average of its scalar curvature $R$ and $\lambda$ is the first eigenvalue for the Laplacian on $M$. The equality sign in (i) holds if and only if $(M,g)$ is Einstein. If $n\ge 4$, then the equality sign in (ii) holds if and only if $(M,g)$ is locally conformally flat. Both (i) and (ii) become equalities if and only if $(M,g)$ has constant curvature.
\end{theorem}
\begin{proof}
 The Riemannian curvature tensor has the following orthogonal decomposition (e.g. \cite[p.26]{chow2006hamilton}):
  \begin{equation}\label{eq: decomp}
  Rm=\frac R{2n(n-1)}g\odot g+ \frac 1{n-2} g\odot \stackrel\circ{Ric} +W
  \end{equation}
  where $W$ is the Weyl tensor (which vanishes when $n=3$). As the decomposition is orthogonal, we have
  $$ \frac 1{n-2}|g\odot \stackrel\circ {Ric}|^2 = \langle Rm- \frac R{2n(n-1)}g\odot g, g\odot \stackrel\circ {Ric}\rangle.$$
  It is easy to compute that $ |g\odot \stackrel\circ {Ric}|^2 =4(n-2)|\stackrel\circ {Ric}|^2$. Thus
  \begin{equation}\label{ineq: pf}
  \begin{split}
    |\stackrel\circ {Ric}|^2
    &=\frac 14 \langle Rm-\frac R{n(n-1)}B, g\odot \stackrel\circ{Ric}\rangle\\
    &\le \frac 14 |Rm-\frac R{n(n-1)}B||g\odot  \stackrel\circ{Ric}|\\
    &=  \frac{\sqrt{n-2}}2 |Rm-\frac R{n(n-1)}B ||\stackrel\circ{Ric}|.
  \end{split}
\end{equation}
We conclude that
$$ |\stackrel\circ {Ric}|^2\le \frac{n-2}4 |Rm-\frac R{n(n-1)}B|^2.$$
On the other hand, from \cite{cheng2012generalization}, we have
\begin{equation*}
\int _M (R-\overline R) ^2 \le \frac {4n(n-1) } {(n-2)^2}(1+\frac{nK}\lambda) \int_M |\stackrel\circ {Ric}|^2.
\end{equation*}
Combining these two inequalities, we can get the result.
If the inequality in (i) becomes an equality, then by \cite{cheng2012generalization}, $(M,g)$ is Einstein. The equality in (i) clearly holds if $(M,g)$ is Einstein.

Now suppose the equality in (ii) holds. Then from \eqref{ineq: pf}, we deduce that
$ Rm - \frac {R}{2n(n-1)}g\odot g=g\odot \stackrel\circ{Ric}+W$ and $g\odot \stackrel\circ {Ric}$ must be linearly dependent. We deduce that $W=0$ and thus $(M,g)$ must be locally conformally flat if $n\ge 4$ (see e.g. \cite{chow2006hamilton} Proposition 1.62). If $(M,g)$ is locally conformally flat, then from \eqref{ineq: pf} we can also deduce that (ii) is an equality.

Finally if both (i) and (ii) become equalities, then $R=\overline R$ is constant by \cite{cheng2012generalization} and thus $Rm$ is also constant by \eqref{ineq: R}. The converse is clear.
\end{proof}

\begin{corollary}\label{cor: 1}
  With the same assumptions as in Theorem \ref{thm: R}, we have
  \begin{equation}\label{ineq: cor}
   \int_M |Rm-\frac{\overline R}{n(n-1)}B|^2 \le \frac{n}{n-2}(1+\frac{2K}{\lambda})\int_M |Rm-\frac{R}{n(n-1)}B|^2.
  \end{equation}
  The equality holds if and only if $(M,g)$ has constant curvature $\frac{\overline R}{n(n-1)}$.
\end{corollary}

\begin{proof}
As $Rm$ has the orthogonal decomposition \eqref{eq: decomp}, we have
$$ \langle Rm- \frac {R}{n(n-1)}B, B\rangle=0.$$
Since
$$Rm- \frac {\overline R}{n(n-1)}B=(Rm- \frac { R}{n(n-1)}B)+\frac {1}{n(n-1)}(R-\overline R)B, $$
by Pythagoras theorem we have
$$ |Rm- \frac {\overline R}{n(n-1)}B|^2=|Rm- \frac { R}{n(n-1)}B|^2+\frac 2{n(n-1)}(R-\overline R)^2.$$
Combining this with \eqref{ineq: R}, we can get the result.
\end{proof}

In a forthcoming paper \cite{kwong2012almost}, we will show that indeed an analogous result similar to Theorem \ref{thm: R} holds for the so called $(p,q)$-curvature (cf. \cite{labbi2005double}) on a Riemannian manifold $(M,g)$.

\subsection{Optimality of the constants}
We remark that the constants in \eqref{ineq: R} and \eqref{ineq: cor} are optimal, which follows directly from the argument of De Lellis and Topping \cite{lellis2012almost}.

 Let $M=(\mathbb S^n , g_0)$ be the sphere with standard metric. In this case, $Ric = (n-1)g_0$ and thus we can choose $K$ in Theorem \ref{thm: R} to be zero. Then \eqref{ineq: R} (i) becomes
 $$ \int _M (R-\overline R) ^2\le C_1\int_M |Ric-\frac Rn g|^2 $$
 where $C_1=\frac {4n(n-1) } {(n-2)^2}$. Using a second variation argument, De Lellis and Topping \cite{lellis2012almost} showed that the constant $C_1$ is optimal in the following sense: for any $0<\alpha<1$, there exists a sufficiently small $t>0$ and a function $f$ on $M$ so that $g_t=(1+tf )g_0$ is a Riemannian metric with
 $$ \int _{M} (R(g_t)-\overline R(g_t)) ^2 dv_{g_t}> \alpha C_1 \int_{M} |Ric(g_t)-\frac {R(g_t)}n g_t|^2 dv_{g_t}.$$
Note that $g_t$ is conformal with $g_0$ and in particular is conformally flat. From this we see that the Weyl tensor $W(g_t)$ vanishes and in particular \eqref{ineq: pf} implies that
$$\alpha C_1 \int_{M} |Ric(g_t)-\frac {R(g_t)}n g_t|^2 dv_{g_t}= \alpha C_2 \int_M |Rm(g_t)-\frac {R(g_t) }{n(n-1)} B|^2dv_{g_t},$$
where $C_2= \frac {n(n-1) } {n-2} $. From this we see that the constants in \eqref{ineq: R} are optimal.

\bibliographystyle{amsplain}

\begin{thebibliography}{10}

\bibitem{besse2008einstein}
A.L. Besse.
\newblock {\em Einstein manifolds}, volume~3.
\newblock Springer Verlag, 2008.


\bibitem{C}
X. Cheng
{\sl An almost-Schur type lemma for symmetric $(2,0)$ tensors and applications}, Arxiv preprint arXiv:1208.2152 (2012).


\bibitem{cheng2012generalization}
X.~Cheng.
\newblock A generalization of almost-schur lemma for closed riemannian
  manifolds.
\newblock {\em Annals of Global Analysis and Geometry}, p. 1--8, 2012.


\bibitem{CZ}
X. Cheng and D. Zhou
{\sl Rigidity for nearly umbilical hypersurfaces in
space forms}, Arxiv preprint arXiv:1208.1786 (2012).



\bibitem{chow2006hamilton}
B.~Chow, P.~Lu, and L.~Ni.
\newblock {\em Hamilton's Ricci flow}, volume~77.
\newblock American Mathematical Society, 2006.

\bibitem{de2005optimal}
C.~De~Lellis and S.~M{\"u}ller.
\newblock Optimal rigidity estimates for nearly umbilical surfaces.
\newblock {\em Journal of Differential Geometry}, 69(1):075--110, 2005.

\bibitem{lellis2012almost}
C.~De~Lellis and P.M. Topping.
\newblock Almost-schur lemma.
\newblock {\em Calculus of Variations and Partial Differential Equations},
  p. 1--8, 2012.

\bibitem{geproblems}
Y.~Ge, G.~Wang, and C.~Xia.
\newblock On problems related to an inequality of andrews, de lellis and
  topping. Preprint.

\bibitem{grosjean2002upper}
J.F. Grosjean,
{\it  Upper bounds for the first eigenvalue of the Laplacian on compact submanifolds},
Pacific J. Math
\textbf{206} (no. 1), (2002),  93--112.

\bibitem{kwong2012almost}
K-K. Kwong
{\it  An almost Schur-type theorem for (p,q)-curvatures},
in preparation.

\bibitem{labbi2005double}
M.L. Labbi.
\newblock Double forms, curvature structures and the (p, q)-curvatures.
\newblock {\em Transactions of the American Mathematical Society}, p.
  3971--3992, 2005.




\bibitem{montiel1991compact}
S.~Montiel and A.~Ros.
\newblock Compact hypersurfaces: the alexandrov theorem for higher order mean
  curvatures.
\newblock {\em Differential Geometry, Pitman Monogr. Surveys Pure Appl. Math},
  52:279--296, 1991.


\bibitem{nash1956imbedding}
J.~Nash.
\newblock The imbedding problem for riemannian manifolds.
\newblock {\em Ann. of Math}, 63(1):20--63, 1956.

\bibitem{Perez}
D. Perez, On nearly umbilical hypersurfaces, thesis, 2011.


\bibitem{reilly1973variational}
R.C. Reilly.
\newblock Variational properties of functions of the mean curvatures for
  hypersurfaces in space forms.
\newblock {\em J. Differential Geom}, 8:465--477, 1973.





\end{thebibliography}

\end{document}